\documentclass[10pt,preprint]{amsart}
\usepackage{amssymb,amsfonts,amsthm,amsmath,amscd,relsize}
\usepackage{hyperref}
\usepackage{comment}

\usepackage{pgf,tikz,pgfplots}
\usepackage{xcolor}
\usetikzlibrary{calc}

\def\N{\mathbb N}

\theoremstyle{plain}
\newtheorem*{theorem}{Theorem}

\theoremstyle{definition}

\theoremstyle{remark}
\newtheorem{remark}{Remark}
\newtheorem*{thks}{Thanks}

\begin{document}

\title{Less than Equable Triangles on the Eisenstein lattice}

\author{Christian Aebi and Grant Cairns}

\address{Coll\`ege Calvin, Geneva, Switzerland 1211}
\email{christian.aebi@edu.ge.ch}
\address{Department of Mathematical and Physical Sciences, La Trobe University, Melbourne, Australia 3086}
\email{G.Cairns@latrobe.edu.au}

%\begin{abstract}
%We show that there are only two equable triangles having vertices on the Eisenstein lattice, up to Euclidean motions. 
%\end{abstract}

\maketitle

\section{Introduction} 
In this paper, we are interested in  triangles that are \emph{perimeter-dominant}, i.e.,  their perimeter exceeds their area. 
Recall that a \emph{Heron triangle} is a triangle with integer sides and integer area.  By Dolan's Theorem \cite{do}, 
the classification of perimeter-dominant Heronian triangles contains three exceptional examples and four  infinite families determined by certain Pell, or Pell-like, equations.

We will obtain an analogous result for triangles with side lengths in $\sqrt3\N$ and area in  $\frac{\sqrt3}4\N$. As we will consider in Section 3 of this paper, such triangles can be considered to be the equivalent of Heron triangles on the Eisenstein lattice.

\section{The Main Theorem} 

\begin{theorem}\label{T:main}
Perimeter-dominant triangles with side lengths in  $\sqrt3\N$ and area in  $\frac{\sqrt3}4\N$ consist of the equilateral 
triangle of side length $3\sqrt3$ and three infinite families of triangles with side lengths:
\begin{enumerate}
\item[\textup{(a)}] $x\sqrt3,x\sqrt3,\sqrt3$, where  $1=4x^2-3y^2$ for some $x,y\in\N$.
\item[\textup{(b)}] $x\sqrt3,x\sqrt3,2\sqrt3$, where  $1= x^2-3y^2$ for some $x,y\in\N$.
\item[\textup{(c)}]  $(3x+1)\sqrt3,(3x-1)\sqrt3,3\sqrt3$, where  $1=4x^2-15y^2$ for some $x,y\in\N$. 
\end{enumerate}
\end{theorem}

\begin{proof}
Let a perimeter-dominant triangle $T$ have side lengths $a\sqrt3, b\sqrt3, c\sqrt3$ and area $\Delta:=\frac{n\sqrt3}4$, where $a,b,c,n\in\N$. 
Heron's formula 
 \cite[Chap.~6.7]{OW} for the area gives
\[
\Delta=\sqrt{s(s-a\sqrt3)(s-b\sqrt3)(s-c\sqrt3)},
\]
where $s=\frac{\sqrt3(a+b+c)}2$ is the semi-perimeter.
Thus, as $T$ is perimeter-dominant,
\[
\frac{n\sqrt3}4=\frac34\sqrt{(a+b+c)(-a+b+c)(a-b+c)(a+b-c)}< \sqrt3(a+b+c),
\]
so
\begin{equation}\label{E:heron}
n^2=3(a+b+c)(-a+b+c)(a-b+c)(a+b-c)<16(a+b+c)^2.
\end{equation}
Let $p=a+b+c, u=-a+b+c, v=a-b+c, w=a+b-c$. Note that the integers $p,u,v,w$ are all positive and have the same parity. (However, unlike the Heron case,  they may be even or odd). 
Without loss of generality, suppose $u\le v\le w$.
Note that $u+v+w=p$, so by \eqref{E:heron}, $3uvw<16(u+v+w)$ and hence
\begin{equation}\label{E:xyz}
\frac1{uv}+\frac1{vw}+\frac1{wu}>\frac3{16}.
\end{equation}
Since $u\le v\le w$, \eqref{E:xyz} gives $\frac{3}{u^2}>\frac3{16}$, and hence $u\le 3$.
For each of these three values of $u$, we can use \eqref{E:xyz} again to get an upper bound on $v$. Then, using the fact that $u,v$ have the same parity, for each possible pair $u,v$ we can use \eqref{E:xyz} again, in some cases obtaining an upper bound on $w$.
Table~\ref{T:poss} gives the potential  values of $u,v,w$ and the corresponding values of $n^2=3(u+v+w)uvw$. Certain potential 
values of $u,v,w$ are eliminated because the resulting values of $n^2$ cannot be squares.
\begin{comment}
If $u=1$, then $v\le w$ and \eqref{E:xyz} gives  $v\le 11$. 
Furthermore, when $u=1,v=9$,  \eqref{E:xyz} gives  $w\le14$, and when $u=1,v=11$,  \eqref{E:xyz} gives  $w=11$. 
If $u=2$, then $v\le w$ and \eqref{E:xyz} gives  $v\le 6$.
Furthermore, when $u=2,v=4$,  \eqref{E:xyz} gives  $w\le 11$, and when $u=2,v=6$,  \eqref{E:xyz} gives  $w=6$. 
If $u=3$, then $v\le w$, \eqref{E:xyz} and the fact that $u,v$ have the same parity gives  $v\le 3$.
\end{comment}

\begin{table}[h!]
\centering
\begin{tabular}{c|c|c|c|c}
  \hline
   $u$ & $v$ & $w$  & $n^2$   & Comment\\\hline
  $1$& $1$ & $w$ odd  & $3w(w+2)$   &\\
  $1$& $3$ & $w$ odd  & $9w(w+4)$   &Not square mod $8$\\
  $1$& $5$ & $w$ odd  & $15w(w+6)$   &\\
  $1$& $7$ & $w$ odd, $7\le w\le 25$  & $21w(w+8)$   &Not square mod $8$\\
  $1$& $9$ & $w$ odd, $9\le w\le 13$  & $27w(w+10)$  &Not square\\
  $1$& $11$ & $11$  & $33\cdot11\cdot23$   &Not square\\
  $2$& $2$ & $w$ even  & $12w(w+4)$   &\\
  $2$& $4$ & $w$ even, $4\le w\le 10$  & $24w(w+6)$   &Not square\\
  $2$& $6$ & $6$   & $6^3\cdot 14$   &Not square \\
  $3$& $3$ & $w$ odd, $3\le w\le 7$  & $27w(w+6)$   & Square only for $w=3$\\
\hline
\end{tabular}

\bigskip
\caption{Possible  values.}\label{T:poss}
\end{table}

The side lengths are given by $a= \frac{v + w}2, b = \frac{u + w}2, c = \frac{u + v}2$. The case $u=v=w=3$ gives $a=b=c=3$, which is the equilateral triangle. The remaining three cases from Table~\ref{T:poss} are $(u,v)=(1,1),(1,5)$ and $(2,2)$.

If $u=v=1$,  then $w$ is odd and $n^2=3w(w+2)$. So $n$ is divisible by $3$. Setting  $n=3y$ and $w= 2x-1$ gives $1=4x^2-3y^2$; this is case (a) of the Theorem. Similarly, if $u=1,v=5$, we have $n^2=15w(w+6)$ and setting  $n=45y$ and $w= 6x-3$ gives $1=4x^2-15y^2$; this is case (c).
Finally, if $u=v=2$,  we have $n^2=12w(w+4)$ and setting  $n=12y$ and $w=2x-2$  gives $1=x^2-3y^2$, which is case (b).
\end{proof}

\begin{remark} 
Table \ref{T} gives the perimeters and areas of the triangles of the Theorem. In cases (a) and (c) of the Theorem, $y$ is necessarily odd. So, from Table~\ref{T}, it is only in case (b) that the area is an integer multiple of $\sqrt3$. 

\begin{table}[h]
\begin{tabular}{c|l|c|c|c}
  \hline
Case&  Relation& Side Lengths/$\sqrt3$ &  Perimeter/$\sqrt3$&Area/$\sqrt3$  \\\hline
 & &$(3,3,3)$ &  $9$& ${27}/4$\\
a &$1=4x^2-3y^2$&   $(x,x,1)$ & $2x+1$&${3y}/4 $\\
b &$1= x^2-3y^2$ & $(x,x,2)$& $2x+2$&$3y$\\
c &  $1=4x^2-15y^2$ &$(3x+1,3x-1,3)$ & $6x+3$&${45y}/4$\\
\hline
\end{tabular}
\bigskip
\caption{Perimeters and areas}\label{T}
\end{table}
\end{remark}

\begin{remark} The three equations of the Theorem are well known. For the equation $1=4x^2-3y^2$ of case (a), the solutions $(x,y)$ are given by the solutions of the Pell equation $ 1=X^2 - 3y^2 $ for which $X$ is even, $X=2x$; see sequence A094347 in \cite{OEIS}. The $x$-values are given by the recurrence relation $x_n = 14x_{n-1} - x_{n-2}$ with $x_1 = 1,x_2 = 13$. The first five $x$-values are $1,13,181,2521,35113$.

For the Pell equation $1=x^2-3y^2$ of case (b), the  $x$-values are given by the recurrence relation $x_n = 4x_{n-1} - x_{n-2}$ with $x_1 = 2,x_2 = 7$; see sequence A001075 in \cite{OEIS}. The first five $x$-values are $2,7,26,97,362$.

For the equation $1=4x^2-15y^2$ of case (c), the solutions $(x,y)$ are given by the solutions of the Pell equation $1= X^2 - 15y^2 = 1$ for which $X$ is even, $X=2x$. The  $x$-values are given by the recurrence relation $x_n = 62x_{n-1} - x_{n-2}$ with $x_1 = 2,x_2 = 122$; for the sequence $\frac12x_n$, see sequence A302329 in \cite{OEIS}.
The first five $x$-values are $2,122,7562,468722,29053202$.
\end{remark}

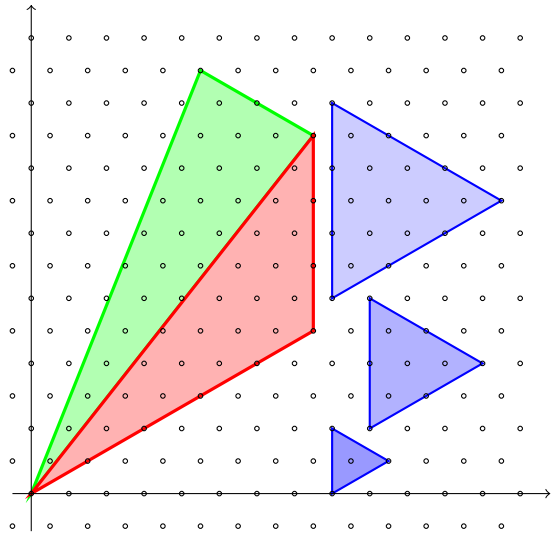
\begin{figure}[h!]
\begin{tikzpicture}[scale=.5]
\def\r{1.732};

\draw [fill, green!30] (0,0) 
  -- (2+11/2, 11*\r/2)  
  -- (-2+13/2, 13*\r/2) 
  -- cycle;

\draw [very thick, green] (0,0) 
  -- (2+11/2, 11*\r/2)  
  -- (-2+13/2, 13*\r/2)  
  -- cycle;

\draw [fill, red!30] (0,0) 
  -- (2+11/2, 11*\r/2)  
  -- (5+5/2,5*\r/2) 
  -- cycle;

\draw [very thick, red] (0,0) 
  -- (2+11/2, 11*\r/2)  
  -- (5+5/2,5*\r/2)  
  -- cycle;

\draw [fill, blue!20] ($(8,0) +(0,3*\r)$)
  -- ($(8+3+3/2, 3*\r/2) +(0,3*\r)$)
  -- ($(8,3*\r)+(0,3*\r)$)
  -- cycle;

\draw [thick, blue] ($(8,0) +(0,3*\r)$)
  -- ($(8+3+3/2, 3*\r/2) +(0,3*\r)$)
  -- ($(8,3*\r)+(0,3*\r)$)
  -- cycle;

\draw [fill, blue!30] ($(9,0) +(0,1*\r)$) 
  -- ($(9+2+2/2, 2*\r/2) +(0,1*\r)$) 
  -- ($(9,2*\r)+(0,1*\r)$) 
  -- cycle;

\draw [thick, blue] ($(9,0) +(0,1*\r)$) 
  -- ($(9+2+2/2, 2*\r/2) +(0,1*\r)$) 
  -- ($(9,2*\r)+(0,1*\r)$) 
  -- cycle;

\draw [fill, blue!40] (8,0) 
  -- (8+1+1/2, 1*\r/2) 
  -- (8,1*\r)
  -- cycle;

\draw [thick, blue] (8,0) 
  -- (8+1+1/2, 1*\r/2) 
  -- (8,1*\r)
  -- cycle;

\foreach \i in {0,...,13}
\foreach \j in {0,...,7}
\draw (\i,\j*\r) circle (.06);
\foreach \i in {0,...,13}
\foreach \j in {-1,...,6}
\draw (\i-1/2,\r*\j+\r/2) circle (.06);
%\draw (\i-1/2,\r*(2*\j+1)/2) circle (.06);

 \draw [->] (0,-1) -- (0,13);
 \draw [->] (-1/2,0) -- (13.8,0);

  \end{tikzpicture}
\caption{Five triangles from the Theorem}\label{F}
\end{figure}

\section{The Eisenstein Lattice} 

Heron triangles can all be realized on the integer lattice \cite{Yiu}, i.e., the lattice of points in the plane with integer coordinates. This is not the case for the triangles studied in this note, since triangles on the integer lattice have integer or half-integer area. However, they can be realised on the 
 \emph{Eisenstein lattice}. The Eisenstein lattice, also known as the triangle lattice or the hexagonal lattice, is the set of vertices of the usual  tiling of the plane by equilateral triangles of side length $1$. Using complex numbers, the Eisenstein lattice is the additive subgroup of $\mathbb C$  generated by the elements $1$ and $\omega=-\frac{1}2+i\frac{\sqrt3}2$. It is shown in \cite{AC2} that all triangles with side lengths in  $\sqrt3\N$ and area in  $\frac{\sqrt3}4\N$ can be drawn on the Eisenstein lattice. Figure~\ref{F} shows five of the triangles of the Theorem. These are the three equilateral triangles (triangle of side length $3\sqrt3$, case (a) with $x=1$ and case (b) with $x=2$), the isosceles triangle of case (b) with $x=7$, and the triangle with sides $(3\sqrt3,5\sqrt3,7\sqrt3)$ of case (c) with $x=2$.
The results in the present note are an extension of those of \cite{AC}, in which we studied triangles whose vertices lie on the Eisenstein lattice and which are \emph{equable}, i.e.,  they have equal perimeter and area.

%%%%%%%%%%%%%%%%%%%%%%%%%

\begin{thks} The authors are very grateful to the referee whose thoughtful suggestions significantly improved the presentation of the paper and simplified the proof.
\end{thks}

%%%%%%%%%%%%%%%%%%%%%%%%%%%

\bibliographystyle{amsplain}
{}

\end{document}